\RequirePackage{fix-cm}

\documentclass[smallextended]{svjour3}
\usepackage{graphics,a4}
\usepackage{amsfonts,color,pslatex}
\usepackage{amssymb,amsmath,latexsym}
\usepackage{float}
\smartqed  
\usepackage{graphicx}

\newtheorem{result}{Result}

\newcommand{\tr}{{\rm Tr}}

\newcommand{\ef}{\mathbb{F}}

\def\tr{\mathrm{Tr}}

\makeatletter
\newcommand{\figcaption}{\def\@captype{figure}\caption}
\newcommand{\tabcaption}{\def\@captype{table}\caption}
\makeatother

\begin{document}

\title{Strongly Regular Graphs Constructed from $p$-ary Bent Functions
}

\author{Yeow Meng Chee  \and Yin Tan \and Xian De Zhang}

\institute{Y. M. Chee\and Y. Tan \and X. Zhang \at Division of Mathematical Sciences, School of Physical and Mathematical Sciences,
         Nanyang Technological University, 21 Nanyang Link Singapore 637371, \email{itanyinmath@gmail.com}
         \\
         Y. M. Chee, \email{ymchee@ntu.edu.sg}\\
         X. Zhang, \email{xiandezhang@ntu.edu.sg}\\
         Y. Tan is currently with Temasek Laboratories, National University of Singapore, 5A Engineering Drive 1,
         $\sharp$09-02, Singapore 117411.}

\date{Received: date / Accepted: date}

\maketitle

\begin{abstract}
    In this paper, we generalize the construction of strongly regular graphs in
	[Y. Tan et al., Strongly regular graphs associated with ternary bent functions, J. Combin.
    Theory Ser. A (2010), 117, 668-682] from
    ternary bent functions to $p$-ary bent functions, where $p$ is an odd prime. We obtain
    strongly regular graphs with three types of parameters.
	Using certain non-quadratic $p$-ary bent functions, our constructions can give rise to
	new strongly regular graphs for small parameters.
\end{abstract}

\section{Introduction}
Boolean bent functions were first introduced by Rothaus in 1976 in
\cite{bent}. They have been extensively studied for their
important applications in cryptography. Such functions have the
maximum Hamming distance to the set of all affine functions. In
\cite{kumar-scholtz-welch}, the authors generalized the notion of
a bent function to  be defined over a finite field of arbitrary
characteristic. Precisely, let $f$ be a function from
$\mathbb{F}_{p^n}$ to $\mathbb{F}_p$. The \textit{Walsh transform}
of $f$ is the complex valued function
$\mathcal{W}_f:\mathbb{F}_{p^n}\rightarrow \mathbb{C}$ defined by
$$\mathcal{W}_f(b):=\sum_{x\in\mathbb{F}_{p^n}}\zeta_p^{f(x)+\tr(bx)}, \quad b\in\mathbb{F}_{p^n},$$
where $\zeta_p$ is a primitive $p$-th root of unity and $\tr(x)$ is the absolute trace function, i.e. $\tr(x):=\sum_{i=0}^{n-1}x^{p^i}.$
The function $f$ is called \textit{$p$-ary bent} if every Walsh coefficient $\mathcal{W}_f(b)$
has magnitude $p^{n/2}$, i.e. $|\mathcal{W}_f(b)|=p^{n/2}$ for all $b\in \mathbb{F}_{p^n}$. Moreover, $f$ is called
\textit{regular} if there exists some function $f^*:\mathbb{F}_{p^n}\rightarrow\mathbb{F}_p$ such that
$\mathcal{W}_f(b)=p^{n/2}\zeta_p^{f^*(b)}$, and  $f$ is called \textit{weakly regular} if
$\mathcal{W}_f(b)=\mu p^{n/2}\zeta_p^{f^*(b)}$ for  some constant $\mu\in\mathbb{C}$ with $|\mu|=1$.
Obviously, regularity implies weak regularity.

It is shown in
\cite{helleseth-kholosha-monobent,helleseth-hollmann-kholosha-wang-xiang}
that quadratic bent functions and most monomial $p$-ary bent
functions are weakly regular, except one sporadic non-weakly
regular example. Recently, a new family of non-quadratic weakly
regular $p$-ary bent functions has been constructed in
\cite{helleseth-kholosha-binobent} and a new sporadic non-weakly
regular bent function was given. However, there are still few
non-quadratic bent functions known over the field $\ef_{p^n}$ when
$p\ge 5$.

There are many motivations to find more weakly regular bent
functions, especially non-quadratic ones. Recently, it is shown in
\cite{pott-yin-wcc}, \cite{tan-pott-srgbent} that, under some conditions, weakly regular bent
functions can be used to construct certain combinatorial objects,
such as strongly regular graphs and association schemes.
Precisely, let $f:\ef_{3^{2k}}\rightarrow\ef_3$ be a weakly
regular bent function. Define
$$D_i:=\{x: x \in\ef_{3^{2k}} | f(x)=i\}, \quad 0\le i\le 2.$$
It is shown in \cite{tan-pott-srgbent} that $D_0, D_1, D_2$ are all regular partial difference sets.
The Cayley graphs generated by $D_0,D_1,D_2$ in the additive group of $\ef_{3^{2k}}$
are strongly regular graphs. Some non-quadratic bent functions seem to give rise to
new families of strongly regular graphs up to isomorphism (see \cite[Tables 2,3]{tan-pott-srgbent}).

In this paper, we generalize the work in \cite{tan-pott-srgbent} by using $p$-ary bent functions to
construct strongly regular graphs. We show that if $f:\ef_{p^{2k}}\rightarrow\ef_p$ satisfying Condition A
(defined in Section 3), then the subsets
\begin{eqnarray}
\label{graphs}
\begin{array}{lll}
 &&D=\{x\in\ef_{p^{2k}}^*| f(x)=0\}, \\
 &&D_\mathcal{S}=\{x\in\ef_{p^{2k}}^*| f(x)\ \mbox{are non-zero squares}\}, \\
 &&D_\mathcal{S}^\prime=\{x\in\ef_{p^{2k}}^*| f(x)\ \mbox{are squares}\} \text{ and} \\
 &&D_\mathcal{N}=\{x\in\ef_{p^{2k}}^*| f(x)\ \mbox{are non-squares}\} \\
\end{array}
\end{eqnarray}
are regular partial difference sets. Using this construction, it seems that the
$p$-ary bent functions in \cite{helleseth-kholosha-binobent} may give rise to new negative Latin square type strongly regular graphs.
For small parameters, we have verified that the graphs are new.

The paper is organized as follows. In Section 2, we give
necessary definitions and results. The constructions of strongly
regular graphs will be given in Section 3. In Section 4, we
discuss the newness of the graphs obtained.

\section{Preliminaries}

Group rings and character theory are useful tools to study
difference sets. We refer to \cite{passman} for basic facts of
group rings and \cite{lidl-niederreiter} for character theory
on finite fields.

Let $G$ be a multiplicative group of order $v$. A $k$-subset $D$ of $G$ is a $(v,k,\lambda,\mu)$ \textit{partial
difference set} (PDS) if each non-identity element in $D$ can be represented as $gh^{-1}\ (g,h\in D, g \ne h)$
in exactly $\lambda$ ways, and each non-identity element in $G\backslash D$ can be represented as
$gh^{-1}\ (g,h\in D, g \ne h)$ in exactly $\mu$ ways. We shall always assume that the identity element
$1_G$ of $G$ is not contained in $D$. Using the group ring language, a $k$-subset
$D$ of $G$ with $1_G\not\in D$ is a $(v,k,\lambda,\mu)$-PDS if and only if the following
equation holds:
\begin{equation}
\label{PDSequation}
DD^{(-1)}=(k-\mu)1_G+(\lambda-\mu)D+\mu G.
\end{equation}

Combinatorial objects associated with partial difference sets are
strongly regular graphs. A graph $\Gamma$ with $v$ vertices is
called a $(v,k,\lambda,\mu)$ \textit{strongly regular graph} (SRG)
if each vertex is adjacent to exactly $k$ other vertices, any two
adjacent vertices have exactly $\lambda$ common neighbours, and any two non-adjacent vertices have exactly $\mu$
common neighbours.

Given a group $G$ of order $v$ and a $k$-subset $D$ of $G$ with $1_G\not\in D$
and $D^{(-1)}=D$, the
graph $\Gamma=(\textit{V},\textit{E})$ defined as follows is called
the \textit{Cayley graph} generated by $D$ in $G$:

\begin{itemize}
\item[(1)] The vertex set \textit{V} is $G$;
\item[(2)] Two vertices $g,h$ are joined by an edge if and only if
$gh^{-1}\in D$.
\end{itemize}

The following result points out the relationship between SRGs and
PDSs.

\begin{result}[\cite{mapds}]
\label{pdsandsrg} Let $\Gamma$ be the Cayley graph generated by a
$k$-subset $D$ of a multiplicative group $G$ with order
$v$. Then $\Gamma$ is a $(v,k,\lambda,\mu)$ strongly regular graph
if and only if $D$ is a $(v,k,\lambda,\mu)$-PDS with $1_G\not\in
D$ and $D^{(-1)}=D$.
\end{result}

Strongly regular graphs (or partial difference sets) with
parameters $(n^2,r(n+\varepsilon),-\varepsilon n+r^2+3\varepsilon
r,r^2+\varepsilon r)$ are called of \textit{Latin Square type}  if
$\varepsilon=-1$, and  of \textit{negative Latin Square type} if
$\varepsilon=1$. There are many constructions of SRGs of Latin
square type (any collection of $r-1$ mutually orthogonal Latin
squares gives rise to such a graph, see \cite{mapds}, for
instance), but only a few constructions of negative Latin square
type are known. We will show that certain weakly regular $p$-ary bent
functions can be used to construct SRGs of Latin square  and of
negative Latin square type.

Next we introduce the concept of association schemes. Let $V$ be a finite set of vertices,
and let $\{R_0,R_1,\ldots,R_d
\}$ be binary relations on $V$ with $R_0:=\{(x,x) : x\in V\}$. The
configuration $(V; R_0,R_1,\ldots,R_d)$ is called an
\textit{association scheme} of class $d$ on $V$ if the following
holds:

\begin{itemize}
\item[(1)] $V\times V=R_0\cup R_1\cup\cdots\cup R_d$ and
$R_i\cap R_j=\emptyset$ for $i\ne j$,
\item[(2)] $^tR_i=R_{i^\prime}$ for some $i^\prime\in\{0,1,\ldots,d\}$,
where $^tR_i:=\{(x,y) |(y,x)\in R_i\}$. If $i^\prime=i$, we call
$R_i$ is \textit{symmetric},
\item[(3)] For $i,j,k\in \{0,1,\ldots,d\}$ and for any pair $(x,y)\in R_k$, the number
$|\{z\in V\: |\:(x,z)\in R_i, (z,y)\in R_j|$ is a constant, which is denoted by
$p_{ij}^k$.
\end{itemize}
An association scheme is said to be \textit{symmetric} if every $R_i$ is symmetric.

Given an association scheme $(V; \{R_l\}_{0\le l\le d})$, we can
take the union of classes to form graphs with larger sets (this is
called \textit{fusion}), but it is not necessarily guaranteed that
the fused collection of graphs will form an association scheme on
$V$. If an association scheme has the property that any of its
fusions is also an association scheme, then we call the
association scheme \textit{amorphic}. Van Dam \cite{dam} proved
the following result.
\begin{result}
\label{amorphic} Let $V$ be a set of size $v$, and let
$\{G_1,G_2,\ldots,G_d\}$ be an edge-decomposition of the complete
graph on $V$, where each $G_i$ is a strongly regular graph on $V$.
If $G_i, 1\le i\le d,$ are all of Latin square type or all of
negative Latin square type, then the decomposition is a $d$-class
amorphic association scheme on $V$.
\end{result}

We conclude this section by recording the bent function in
\cite{helleseth-kholosha-binobent} as below.
\begin{result}
\label{binomialbent}
Let $n=4k$. Then the $p$-ary function $f(x)$ mapping $\ef_{p^n}$ to $\ef_p$ given by
$$f(x)=\tr_n(x^2+x^{p^{3k}+p^{2k}-p^k+1})$$
is a weakly regular bent function. Moreover, for $b\in\ef_{p^n}$ the corresponding Walsh
coefficient of $f(x)$ is equal to
$$\mathcal{W}_f(b)=-p^{2k}\zeta_p^{\tr_k(x_0)/4},$$
where $x_0$ is a unique solution in $\ef_{p^k}$ of the equation
$$b^{p^{2k}+1}+(b^2+x)^{(p^{2k}+1)/2}+b^{p^k(p^{2k}+1)}+(b^2+x)^{p^k(p^{2k}+1)/2}=0.$$
\end{result}

\section{The construction}
In this section, we construct SRGs using $p$-ary bent functions.
First we introduce some notations used throughout this section.
Let $f:\ef_{p^n}\rightarrow \ef_p$ be a weakly regular $p$-ary
bent function satisfying $f(-x)=f(x)$. Without loss of generality,
we may assume $f(0)=0$. If not, we can replace $f(x)$ with
$f(x)-f(0)$. For each $b\in\ef_{p^n}$, assume that
$\mathcal{W}_f(b)=\mu (\sqrt{p^*})^n\zeta_p^{f^*(b)}$, where
$\mu=\pm 1$ and $p^*=(-1)^{\frac{p-1}{2}}p$. Suppose that there
exists an integer $l$ with $(l-1,p-1)=1$ such that for each
$\alpha\in \ef_p$ and $x\in\ef_{p^n}$, $f(\alpha x)=\alpha^lf(x)$
holds. Let
$$D_i:=\{x\in\ef_{p^n}\: |\: f(x)=i\}, \quad 0\le i\le p-1.$$
Denote by $G$ and $H$ the additive groups of $\ef_{p^n}$ and
$\ef_p$ respectively. Clearly we have
$\sum\limits_{i=0}^{p-1}D_i=G$. Moreover, $D_i^{(-1)}=D_i$ for
each $0\le i\le p-1$ since $f(-x)=f(x)$. Define the group ring
elements in $\mathbb{Z}[\zeta_p](G)$:
$$L_t:=\sum_{i=0}^{p-1}D_i\zeta_p^{it},\; 0\le t\le p-1,$$
in particular $L_0=\ef_{p^n}$. The following result gives some properties of the $L_t$'s (see
\cite{pott-yin-wcc}).
\begin{result}
\label{lilj} $(1)$ If $s,t,s+t\ne 0$, then
$L_tL_s=\mu(\frac{tsv}p)^n(\sqrt{p^*})^nL_v$ with
$s^{1-l}+t^{1-l}=v^{1-l}$;

$(2)$ $L_tL_{-t}=p^n$ for $t\in\{1,\ldots,p-1\}$;

$(3)$ $\sum\limits_{t=1}^{p-1}L_tL_0\zeta_p^{-at}=(p|D_a|-p^n)\ef_{p^n}$.
\end{result}

It is clear that we may compute $D_i$'s from $L_t$'s, namely
$D_i=\frac{1}p\sum\limits_{t=0}^{p-1}L_t\zeta_p^{-it}$.
By Result \ref{lilj}, we have
\begin{equation}
\begin{array}{lll}
    \label{dadb}
    p^2D_aD_b&=&\sum\limits_{s,t=0}^{p-1}L_tL_s\zeta_p^{-at-bs}       \\                                 \\
             &=&L_0^2+\sum\limits_{s, t\ne 0\atop s+t\ne 0}L_tL_s\zeta_p^{-at-bs}
            +\sum\limits_{s=1}^{p-1}L_sL_{-s}\zeta_p^{s(a-b)}
            +\sum\limits_{s=1}^{p-1}L_sL_0\zeta_p^{-bs}+\sum\limits_{t=1}^{p-1}L_tL_0\zeta_p^{-at}           \\
            &=&p^n\ef_{p^n}+\sum\limits_{s, t, s+t\ne 0\atop s^{1-l}+t^{1-l}=v^{1-l}}
             \mu(\frac{tsv}p)^n(\sqrt{p^*})^n\zeta_p^{-at-bs}L_v+\sum\limits_{s=1}^{p-1}p^n\zeta_p^{s(a-b)}
             +(p|D_b|-p^n)\ef_{p^n}\\&&+(p(|D_a|-p^n)\ef_{p^n} \\
            &=&\sum\limits_{s=1}^{p-1}p^n\zeta_p^{s(a-b)}+(p(|D_a|+|D_b|)-p^n)\ef_{p^n}
           +\sum\limits_{s, t, s+t\ne 0\atop s^{1-l}+t^{1-l}=v^{1-l}}
            \mu(\frac{tsv}p)^n(\sqrt{p^*})^n\zeta_p^{-at-bs}L_v.\\
\end{array}
\end{equation}

In the following we only work on the field $\ef_{p^n}$ with
$n=2k$. Note that in this case the Walsh coefficient of $f$ can be
written as the form $\mathcal{W}_f(b)=(-1)^{\frac{(p-1)k}{2}}\mu
p^k\zeta_p^{f^*(b)}$. For each $b\in\ef_{p^{2k}}$, let $\chi_b$ be
the additive character of $G$ defined by
$\chi_b(x)=\zeta_p^{\tr(bx)}$, and  $\eta$ be the additive
character of $H$ defined by $\eta(x)=\zeta_p^x$. Now for each
$b\in\ef_{p^n}$,
$$
\mathcal{W}_f(b)=\sum_{x\in\ef_{p^{2k}}}\zeta_p^{f(x)+\tr(bx)}
                 =\sum_{i=0}^{p-1}(\sum\limits_{x\in D_i}\zeta_p^{\tr(bx)})\zeta_p^i
                 =\sum_{i=0}^{p-1}\chi_b(D_i)\zeta_p^i
                 =\chi_b\eta(R),$$
where $R=\{(x,f(x)): x \in\ef_{p^{2k}}\}$. Since $f$ is a bent function, we know that
$|\chi_b\eta(R)|=|\mathcal{W}_f(b)|=p^{k}$. \\

First we determine the cardinalities of $D_i$'s.

\begin{lemma}
    \label{D0size}
    Let $f:\ef_{p^{2k}}\rightarrow\ef_p$ be the bent function as above. Then

    $(1)$ $|D_1|=|D_2|=\cdots=|D_{p-1}|$,

    $(2)$ $|D_0|=p^{2k-1}+\epsilon (p^k-p^{k-1})$ and $|D_i|=p^{2k-1}-\epsilon p^{k-1}$ for each $1\le i\le p-1$,
    where $\epsilon=(-1)^{\frac{(p-1)k}{2}}\mu$.
\end{lemma}
\begin{proof}
    (1)  For any $1\le a,b\le p-1$, by Result \ref{lilj} (3) we
    have
$$
    (p|D_a|-p^{2k})\ef_{p^{2k}}=\sum_{t=1}^{p-1}L_tL_0\zeta_p^{-at}
    =\sum_{t=1}^{p-1}L_tL_0(\zeta_p^{ab^{-1}})^{-bt}=(p|D_b|-p^{2k})\ef_{p^{2k}}.
    $$ The last equality holds since $\zeta_p^{ab^{-1}}$ is also a primitive $p$-th root of
    unity.
    Thus $|D_a|=|D_b|$.

    (2) Let $\chi_0$ be the principal character of $\ef_{p^{2k}}$, we have
    $$
    \chi_0\eta(R)=|D_0|+\sum_{i=1}^{p-1}|D_i|\zeta_p^i=|D_0|+|D_1|(\zeta_p+\zeta_p^2+\cdots+\zeta_p^{p-1})=|D_0|-|D_1|.
    $$
    Since $\chi_b\eta(R)=\mathcal{W}_f(b)=(-1)^{\frac{(p-1)k}{2}}\mu
p^k\zeta_p^{f^*(b)}=\epsilon p^k\zeta_p^{f^*(b)}$ and
$|D_0|-|D_1|$ is a rational integer, we have
    \begin{equation}
        \label{eq1}
        |D_0|-|D_1|=\epsilon p^k.
    \end{equation}
    On the other hand,
    \begin{equation}
        \label{eq2}
        |D_0|+(p-1)|D_1|=p^{2k}.
    \end{equation}
    By solving Eqs. (\ref{eq1}) and (\ref{eq2}), the result follows.
\end{proof}

Next, we define Condition A for a function $f$ as follows.

\textbf{Condition A:}
  Let $f:\ef_{p^{2k}}\rightarrow\ef_p$ be a weakly regular bent function with $f(0)=0$ and $f(-x)=f(x)$,
  where $p$ is an odd prime. There exists an integer $l$ with $(l-1,p-1)=1$ such that
  $f(\alpha x)=\alpha^lf(x)$ for any $\alpha\in\ef_p$ and $x\in\ef_{p^{2k}}$. For each $b\in\ef_{p^{2k}}$,
 $\mathcal{W}_f(b)=\epsilon p^{k}\zeta_p^{f^*(b)}$, where $\epsilon=(-1)^{\frac{(p-1)k}{2}}\mu$ with $\mu=\pm 1$.

Two functions $f, g:\ef_{p^n}\rightarrow\ef_p$ are called \textit{affine equivalent}
if there exist an affine permutation $A_1$ of $\ef_p$ and an affine permutation $A_2$ of $\ef_{p^n}$ such that
$g=A_1\circ f\circ A_2$. Furthermore, they are called \textit{extended affine equivalent (EA-equivalent)}
if $g=A_1\circ f\circ A_2+A$, where $A:\ef_{p^n}\rightarrow\ef_p$ is an affine function.
A polynomial $L$ of the form $L(x)=\sum\limits_{i=0}^{n}a_ix^{p^i}\in\ef_{p}[x]$
is called a \textit{linearized polynomial}. Note that the affine permutations of $\ef_p$ are $cx+d$, where $c\in\ef_p^*$ and $d\in\ef_p$.

We have the following result.

\begin{proposition}
Let $L_1,L_2\in\ef_{p}[x]$ be linearized polynomials, where $L_1$ is a permutation of $\ef_p$ and $L_2$ is a permutation of $\ef_{p^{2k}}$.
If $L_1(1)\ne 0$ and $f$ satisfies Condition A, then the function $g=L_1\circ f\circ L_2$ satisfies Condition A.
\end{proposition}
\begin{proof}
That $g(0)=0, g(-x)=g(x)$ and $g(\alpha x)=\alpha^lg(x)$ for $\alpha\in\ef_p$ are easy to be verified.
We only need to prove that for each $b\in\ef_{p^{2k}}$,
$\mathcal{W}_g(b)=\epsilon p^{k}\zeta_p^{g^*(b)}$, where $\epsilon=(-1)^{\frac{(p-1)k}{2}}\mu$ with $\mu=\pm 1$.
First assume that $L_1(x)=cx$. Note that $L_1(1)\ne 0$ implies that $c\ne 0$ and $\zeta_p^c$ is also a primitive $p$-th root of unity. Now
\begin{eqnarray*}
  \begin{array}{lll}
    \mathcal{W}_g(b)&=&\sum\limits_{x\in\ef_{p^{2k}}}\zeta_p^{L_1(f(L_2(x))+\tr(bx)}=\sum\limits_{x\in\ef_{p^{2k}}}(\zeta_p^c)^{f(L_2(x)+\tr(c^{-1}bx)} \\
                    &=&\sum\limits_{y\in\ef_{p^{2k}}}(\zeta_p^c)^{f(y)+\tr(c^{-1}bL_2^{-1}(y)}=\sum\limits_{y\in\ef_{p^{2k}}}(\zeta_p^c)^{f(y)+\tr((L_2^{-1})^\star(c^{-1}b)y)} \\
                    &=&\mathcal{W}_f((L_2^{-1})^\star(c^{-1}b))=\epsilon p^{k}\zeta_p^{f^*((L_2^{-1})^\star(c^{-1}b))}, \\
  \end{array}
\end{eqnarray*}
where $(L_2^{-1})^\star$ is the adjoint operator of $L_2^{-1}$. We finish the proof.
\end{proof}

\begin{remark}
Clearly the function $g$ in the above Proposition is affine equivalent to $f$. However, for a function $g$ which is
EA-equivalent but not affine equivalent to $f$, it may be seen that $g$ does not satisfy Condition A. Indeed, assume that
$g=A_1\circ f\circ A_2+A$, then $g(\alpha x)\ne \alpha^{l-1}g(x)$ for $\alpha\in\ef_p$ when $l > 2$.
\end{remark}

Now assume that a function $f:\ef_{p^{2k}}\rightarrow\ef_p$ satisfies
Condition A,  then clearly the functions $L_1\circ f\circ L_2$ all satisfy Condition A, where
$L_1,L_2\in\ef_{p}[x]$ are linearized polynomials, and $L_1$ is a permutation of $\ef_p$, $L_2$ is a permutation of $\ef_{p^{2k}}$.
We see that the functions of the form $L_1\circ f\circ L_2$ are affine equivalent to $f$. However, for a function $g$ which is
EA-equivalent but not affine equivalent to $f$, it may be seen that $g$ does not satisfy Condition A.

Now we will prove the first result in this section.

\begin{theorem}
\label{thm1}
    Let $f$ be a function satisfying condition A.
    Let $$D:=\{x: x \in \ef_{p^{2k}}^*| f(x)=0\}.$$ Then $D$ is a $(v,d,\lambda_1,\lambda_2)$-PDS, where
  \begin{equation}
      \label{para1}
      \begin{array}{lll}
        v&=&p^{2k},\\
        d&=&(p^{k}-\epsilon)(p^{k-1}+\epsilon),\\
        \lambda_1&=&(p^{k-1}+\epsilon)^2-3\epsilon(p^{k-1}+\epsilon)+\epsilon p^k,\\
        \lambda_2&=&(p^{k-1}+\epsilon)p^{k-1}.\\
      \end{array}
  \end{equation}
\end{theorem}

\begin{proof}
    Recall that $L_t=\sum\limits_{i=0}^{p-1}D_i\zeta_p^{it}$. By Eq. (\ref{dadb}), we have:
    \begin{equation}
       \label{computation1}
          p^2D_0D_0=(p-1)p^{2k}+(2p|D_0|-p^{2k})\ef_{p^{2k}}
           +\epsilon\sum\limits_{s, t, s+t\ne 0\atop s^{1-l}+t^{1-l}=v^{1-l}}p^kL_v,\\
    \end{equation}
    Now we compute the last term of Eq. (\ref{computation1}).
    \begin{equation}
    \label{lasterm}
    \begin{array}{lll}
    \sum\limits_{s, t, s+t\ne 0\atop s^{1-l}+t^{1-l}=v^{1-l}}p^kL_v&=&p^k\sum\limits_{i=1}^{p-1}(p-2)L_i\\
                                               &=&(p-2)p^k\sum\limits_{i=1}^{p-1}(D_0+\sum\limits_{j=1}^{p-1}D_j\zeta_p^{ji})\\
                                               &=&(p-2)p^k((p-1)D_0+\sum\limits_{j=1}^{p-1}D_j(\sum\limits_{i=1}^{p-1}\zeta_p^{ji}))\\
                                               &=&(p-2)p^k((p-1)D_0-\sum\limits_{j=1}^{p-1}D_j)\\
                                               &=&(p-2)p^k((p-1)D_0-(\ef_{p^{2k}}-D_0))\\ [2ex]
                                               &=&(p-2)p^k(pD_0-\ef_{p^{2k}}).\\
    \end{array}
    \end{equation}
 Substituting Eq. (\ref{lasterm}) in Eq. (\ref{computation1}), we have
    \begin{equation}
    \label{computation2}
    \begin{array}{lll}
     p^2D_0D_0&=&(p-1)p^{2k}+(2p|D_0|-p^{2k})\ef_{p^{2k}}+\epsilon (p-2)p^{k}(pD_0-\ef_{p^{2k}}) \\ [2ex]
         &=&(p-1)p^{2k}+\epsilon p^{k+1}(p-2)D_0+(2p|D_0|- p^{2k}-\epsilon (p-2)p^k)\ef_{p^{2k}}. \\
    \end{array}
    \end{equation}
  Note that $D+0=D_0$ and $|D_0|=p^{2k-1}+\epsilon(p^k-p^{k-1})$ (Lemma \ref{D0size}),
    by Eq. (\ref{computation2}) we have
    $$
    p^2(D+0)^2=(p-1)p^{2k}+\epsilon p^{k+1}(p-2)(D+0)+(2p|D_0|- p^{2k}-\epsilon (p-2)p^k)\ef_{p^{2k}}.
    $$
    After simplifying, we get the equation
    $$
    D^2=(p^{2k-1}-p^{2k-2}+\epsilon p^k-2\epsilon p^{k-1}-1)+(\epsilon p^k-2\epsilon p^{k-1}-2)D+(\epsilon p^{k-1}+p^{2k-2})\ef_{p^{2k}}.
    $$
    By Eq. (\ref{PDSequation}), the proof is done.
\end{proof}

\begin{remark}
    When $\epsilon=-1$ in Theorem \ref{thm1}, we get negative Latin square type SRGs. When $p=3$, some new SRGs
    arise using non-quadratic ternary bent functions; see \cite[Tables 2,3]{tan-pott-srgbent}.
	Unfortunately, when $p\ge 5$, for the known bent functions, we don't get new graphs.
\end{remark}

To give another construction of SRGs using $p$-ary bent functions,
we show two lemmas first.

\begin{lemma}
    \label{stl}
    Let  $S$ and $T$ be the sets of non-zero squares and non-squares in $\ef_p$ respectively, where $p$ is an odd prime.
    Then we have

    $(1)$ when $p\equiv1 \pmod 4$,
    $$
    \begin{array}{lll}
        S^2&=&\frac{p-1}2+\frac{p-5}4S+\frac{p-1}4T, \\
        T^2&=&\frac{p-1}2+\frac{p-1}4S+\frac{p-5}4T, \\
        ST &=&TS=\frac{p-1}4(S+T);                  \\
    \end{array}
    $$

    $(2)$ when $p\equiv3 \pmod 4$,
    $$
    \begin{array}{lll}
        S^2 &=& \frac{p-3}4S+\frac{p+1}4T, \\
        T^2 &=& \frac{p+1}4S+\frac{p-3}4T, \\
        ST  &=& TS=\frac{p+1}4+\frac{p-3}4\ef_p.   \\
    \end{array}
    $$
\end{lemma}
\begin{proof}
    (1) Note that $-1$ is a square when $p\equiv1 \pmod 4$ and $S^{(-1)}=S, T^{(-1)}=T$ in this case.
    By \cite[Theorem 2]{ding-helleseth-ads}, $S$ is a
    $(p,\frac{p-1}2,\frac{p-5}4,\frac{p-1}2)$ almost difference set in $\ef_p$, which implies that
    $$ SS^{(-1)}=S^2=\frac{p-1}2+\frac{p-5}4S+\frac{p-1}4T.$$
    It is easy to see that $T=\ef_p-0-S$, and hence $TT^{(-1)}=T^2=(\ef_p-0-S)^2$. Now the results can be
    followed by direct computation.

    (2) When $p\equiv3 \pmod 4$, $-1$ is a non-square, and $S^{(-1)}=T, T^{(-1)}=S$.
    It is well known that the subset $S$ is a  $(p,\frac{p-1}2,\frac{p-3}4)$ difference set in $\ef_p$. Hence
    $$SS^{(-1)}=ST=\frac{p+1}4+\frac{p-3}4\ef_p.$$
    The computations are similar to those in $(1)$.
\end{proof}

\begin{lemma}
    \label{m(1+m)}
    Let $\zeta_p$ be a primitive $p$-th root of unity,  $S$ and $T$ be the sets of non-zero squares and non-squares of
    $\ef_p$ respectively. Define $m=\sum_{i\in S}\zeta_p^{-i}$, then

    (1) when $p\equiv1 \pmod 4$, $-m(1+m)=-\frac{p-1}4$;

    (2) when $p\equiv3 \pmod 4$, $-m(1+m)=\frac{p+1}4$.
\end{lemma}
\begin{proof}
We only prove the case $p\equiv1 \pmod 4$, the proof of (2) is
similar. Note that $-(1+m)=\sum_{i\in T}\zeta_p^{-i}$, hence
$-m(1+m)=\sum_{i\in S, j\in T}\zeta_p^{-(i+j)}$. By Lemma
\ref{stl} (1), we know that $\sum_{i\in S, j\in
T}\zeta_p^{-(i+j)}=\frac{p-1}4\sum_{i\in\ef_p^*}\zeta_p^{-i}=-\frac{p-1}4$.
\end{proof}

Now we prove the following result.
\begin{theorem}
  \label{thm2}
  Let $f$ be a function satisfying Condition A.
  Let $$D_\mathcal{S}:=\{x: x \in \ef_{p^{2k}}^*| f(x) \ \mbox{are non-zero squares}\},$$ then $D_\mathcal{S}$ is a
  $(v,d,\lambda_1,\lambda_2)$-PDS, where
  \begin{equation}
      \label{para2}
      \begin{array}{lll}
       v&=&p^{2k},\\
       d&=&\frac{1}2(p^k-p^{k-1})(p^k-\epsilon),\\
       \lambda_1&=&\frac{1}4(p^k-p^{k-1})^2-\frac{3\epsilon}2(p^k-p^{k-1})+p^k\epsilon, \\
       \lambda_2&=&\frac{1}2(p^k-p^{k-1})(\frac{1}2(p^k-p^{k-1})-\epsilon).\\
      \end{array}
  \end{equation}
\end{theorem}
\begin{proof}
    We only prove the case $p\equiv1 \pmod 4$, the proof of the case $p\equiv3 \pmod 4$ is similar.
    Let $S$ and $T$ be the sets of non-zero squares and non-squares of
    $\ef_p$ respectively. Now by Eq. (\ref{dadb}), we have
    \begin{equation}
        \label{expand}
        \begin{array}{lll}
            &&p^2D_\mathcal{S}^2=p^2\sum\limits_{a,b\in S}D_aD_b \\
           &&=\sum\limits_{a,b\in S}(\sum\limits_{s=1}^{p-1}p^{2k}\zeta_p^{s(a-b)}
           +(2p|D_1|-p^{2k})\ef_{p^{2k}}+
           \epsilon p^k\sum\limits_{s, t, s+t\ne 0\atop s^{1-l}+t^{1-l}=v^{1-l}}\zeta_p^{-at-bs}L_v). \\
       \end{array}
    \end{equation}
    Now
    \begin{equation}
        \label{firstwoterms}
        \begin{array}{lll}
            &&\sum\limits_{a,b\in S}(\sum\limits_{s=1}^{p-1}p^{2k}\zeta_p^{s(a-b)})=\frac{p^{2k}(p-1)(p+1)}4, \\
            &&\sum\limits_{a,b\in S}(2p|D_1|-p^{2k})\ef_{p^{2k}}=(\frac{p(p-1)^2}2|D_1|-\frac{p^{2k}(p-1)^2}4)\ef_{p^{2k}}. \\
        \end{array}
    \end{equation}
    To compute the third term in Eq. (\ref{expand}), for $s,t\in\ef_p^*$, we define $\delta(s,t)$ as follows
    $$
    \delta(s,t) = \left\{
                  \begin{array}{ll}
                      m^2 &     \mbox{if}\ s,t\ \mbox{both are squares} \\
                      (1+m)^2 & \mbox{if}\ s,t\ \mbox{both are nonsquares} \\
                      -m(1+m)  & \mbox{others},
                  \end{array}
                  \right.
    $$
    where $m=\sum_{i\in S}\zeta_p^{-i}$. For convenience, denote the set $\{(s,t): s,t \in\ef_p | s\ne 0, t\ne 0, s+t\ne 0\}$ by
    $\Omega$. For $s,t,s+t\ne 0$, define the function $\sigma(s,t)=v$, where $s^{1-l}+t^{1-l}=v^{1-l}$.
    Since $(l-1,p-1)=1$, $\sigma$ is well defined. Now we compute the third term of Eq. (\ref{expand}):
    \begin{equation}
        \label{thirdterm}
        \begin{array}{lll}
            &&\sum\limits_{a,b\in S}\sum\limits_{(s, t)\in\Omega}\zeta_p^{-at-bs}L_{\sigma(s,t)}\\[2ex]
            =&&\sum\limits_{(s, t)\in\Omega}(\sum\limits_{a\in S}(\zeta_p^{-t})^a)(\sum\limits_{b\in S}(\zeta_p^{-s})^b)L_{\sigma(s,t)}
            =\sum\limits_{(s, t)\in\Omega}\delta(s,t)L_{\sigma(s,t)}\\[2ex]
            =&&\sum\limits_{(s,t)\in\Omega\cap (S\times S)}m^2L_{\sigma(s,t)}
              +\sum\limits_{(s,t)\in\Omega\cap (T\times T)}(1+m)^2L_{\sigma(s,t)}
              -\sum\limits_{(s,t)\in\atop{\Omega\setminus((S\times S)\cup(T\times T))}}m(1+m)L_{\sigma(s,t)}. \\
        \end{array}
    \end{equation}
    By Lemma \ref{stl}(1), we know that the multiset
    $\{*\ \sigma(s,t): (s,t)\in\Omega\cap (S\times S) \ *\}= \frac{p-5}4S+\frac{p-1}4T$,
    then
    $$\sum\limits_{(s,t)\in\Omega\cap (S\times S)}m^2L_{\sigma(s,t)}=
     m^2(\frac{p-5}4\sum_{v\in S}L_v+\frac{p-1}4\sum_{v\in T}L_v).$$
     Using similar argument, we can compute the last two terms of Eq. (\ref{thirdterm}).
     Then we have
     \begin{equation}
        \label{combine1}
        \begin{array}{lll}
         &&\sum_{a,b\in S}\sum_{(s, t)\in\Omega}\zeta_p^{-at-bs}L_{\sigma(s,t)} \\[2ex]
         =&&m^2(\frac{p-5}4\sum_{v\in S}L_v+\frac{p-1}4\sum_{v\in T}L_v)+
         (1+m)^2(\frac{p-1}4\sum_{v\in S}L_v+\frac{p-5}4\sum_{v\in T}L_v)                     \\[2ex]
         &&-2m(1+m)(\frac{p-1}4\sum_{v\in \ef_p^*}L_v), \\ [2ex]
         =&&A\sum_{v\in S}L_v+B\sum_{v\in T}L_v,              \\[2ex]
        \end{array}
    \end{equation}
    where $A=\frac{p-5}4m^2+\frac{p-1}4(1+m)^2-\frac{p-1}2m(1+m)$ and
    $B=\frac{p-1}4m^2+\frac{p-5}4(1+m)^2-\frac{p-1}2m(1+m)$.
    Now we see that
    \begin{equation}
        \label{combin2}
        \begin{array}{lll}
        \sum\limits_{a,b\in S}\sum\limits_{(s, t)\in\Omega}\zeta_p^{-at-bs}L_{\sigma(s,t)}
        &=&A\sum_{v\in S}L_v+B\sum_{v\in T}L_v \\[2ex]
        &=&A\sum_{v\in S}(\sum_{i=0}^{p-1}D_i\zeta_p^{vi})+B\sum_{v\in T}(\sum_{i=0}^{p-1}D_i\zeta_p^{vi}) \\ [2ex]
        &=&\sum_{i=0}^{p-1}(A\sum_{v\in S}\zeta_p^{vi}+B\sum_{v\in T}\zeta_p^{vi})D_i. \\[2ex]
        \end{array}
    \end{equation}
    Denote $a_i=A\sum_{v\in S}\zeta_p^{vi}+B\sum_{v\in T}\zeta_p^{vi}$ for $0\le i\le p-1$. Clearly,
    $$
    \begin{array}{lll}
    a_0&=&\frac{p-1}2(A+B) \\
       &=&\frac{p-1}2(-2m^2-2m+\frac{p-3}2) \\
       &=&-(p-1)m(1+m)+\frac{(p-1)(p-3)}4                    \\
       &=&-\frac{(p-1)^2}4+\frac{(p-1)(p-3)}4 \\
       &=&-\frac{p-1}2.   \\
    \end{array}
    $$
    Similarly, we get the following:
    $$
    a_i = \left\{
                  \begin{array}{ll}
                      -(p-1)/2\quad     &     i\in\{0\}\cup T, \\
                      (p+1)/2\quad      &     i\in S. \\
                  \end{array}
                  \right.
    $$
    Now
    \begin{equation}
        \label{combine3}
        \begin{array}{lll}
        \sum\limits_{a,b\in S}\sum\limits_{(s, t)\in\Omega}\zeta_p^{-at-bs}L_{\sigma(s,t)}&=&
        -\frac{p-1}2(\ef_{p^{2k}}-D_\mathcal{S})+\frac{p+1}2D_\mathcal{S}. \\
    \end{array}
    \end{equation}
    By Eqs. (\ref{expand}), (\ref{firstwoterms}) and (\ref{combine3}), we know that
    $$
    \begin{array}{lll}
        &&p^2D_\mathcal{S}D_\mathcal{S}^{(-1)}=p^2D_\mathcal{S}^2\\[2ex]
        =&&\frac{p^{2k}(p-1)(p+1)}4+(\frac{p(p-1)^2}2|D_1|-\frac{p^{2k}(p-1)^2}4)\ef_{p^{2k}}
        +\epsilon p^k(-\frac{p-1}2(\ef_{p^{2k}}-D_\mathcal{S})+\frac{p+1}2D_\mathcal{S}) \\[2ex]
        =&&C_1+C_2D_{\mathcal{S}}+C_3\ef_{p^{2k}},\\
    \end{array}
    $$
    where $C_1=\frac{1}4p^{2k}(p-1)(p+1),  C_2=p^{k+1}\epsilon$ and
    $C_3=\frac{1}4p^{2k}(p-1)^2-\frac{1}2p^{k+1}(p-1)\epsilon$.
    The proof follows by Eq. (\ref{PDSequation}).
    \end{proof}

\begin{remark}
    \label{non-square}
(1) Using similar proof, we may also prove that the set
$$D_\mathcal{N}:=\{x: x \in \ef_{p^{2k}}^*| f(x) \ \mbox{are non-squares}\}$$
is a PDS with the same paramters as in Theorem \ref{thm2}.

(2) In \cite{tan-pott-srgbent}, it is shown that for weakly regular ternary bent function
$f:\ef_{3^{2k}}\rightarrow\ef_3$, the set $D_i:=\{x\in\ef_{3^{2k}}^* | f(x)=i\}$
is a partial difference set for each $0\le i\le 2$. We may see that Theorem \ref{thm2}
is the generalization of the result in \cite{tan-pott-srgbent}.
\end{remark}

With a small modification, we may get the following result.
\begin{theorem}
  \label{thm3}
  Let $f$ be a function satisfying Condition A.
  Let $$D_\mathcal{S}^\prime:=\{x: x \in \ef_{p^{2k}}^*| f(x) \ \mbox{are squares}\},$$
  then $D_\mathcal{S}^\prime$ is a  $(v,d,\lambda_1,\lambda_2)$-PDS, where
  \begin{equation}
      \label{para3}
      \begin{array}{lll}
          v=p^{2k},\\
          d=\frac{1}2(p^k+p^{k-1}+2\epsilon)(p^k-\epsilon),\\
          \lambda_1=\frac{1}4(p^k+p^{k-1}+2\epsilon)^2-\frac{3\epsilon}2(p^k+p^{k-1}+2\epsilon)+p^k\epsilon,\\
          \lambda_2=\frac{1}4(p^k+p^{k-1})(p^k+p^{k-1}+2\epsilon).
      \end{array}
  \end{equation}
\end{theorem}
\begin{proof}
    Clearly $D_\mathcal{S}^\prime=D+D_\mathcal{S}$, where $D=\{x:x\in\ef_{p^{2k}}^*|f(x)=0\}$.
    Then $D_\mathcal{S}^\prime(D_\mathcal{S}^\prime)^{(-1)}=(D_\mathcal{S}+D)(D_\mathcal{S}+D)$. The result follows from
    Theorems \ref{thm1}, \ref{thm2} and similar group ring computations as those in Theorem \ref{thm2}.
\end{proof}
\begin{remark}
    By using MAGMA,  we know that
    Theorems \ref{thm1}, \ref{thm2}, \ref{thm3} are not true for non-weakly regular bent function
    $f(x)=\tr(\xi^7x^{98})$ over $\ef_{3^6}$, where $\xi$ is a primitive element of $\ef_{3^6}$.
    This implies that the weakly regular condition is necessary.
\end{remark}

We conclude this section by a result on association schemes.
Combining Result \ref{amorphic} and Theorems
\ref{thm1},\ref{thm2},\ref{thm3}, we have the following result.
\begin{theorem}
    \label{thm4}
    Let $f$ be a function satisfying Condition A.
    Define the following sets:
    $$
    \begin{array}{lll}
    D&=&\{x: x \in \ef_{p^{2k}}^*| f(x)=0\},  \\
    D_\mathcal{S}&=&\{x: x \in \ef_{p^{2k}}^*| f(x)\ \mbox{are non-zero squares}\}, \\
    D_\mathcal{N}&=&\{x: x \in \ef_{p^{2k}}^*| f(x)\ \mbox{are non-squares}\}.
    \end{array}
    $$
    Then $\{\ef_{p^{2k}}; \{0\},D,D_\mathcal{S},D_\mathcal{N}\}$ is an amorphic association scheme of class $3$.
\end{theorem}

\section{Newness}
In this section, we discuss the known constructions of the SRGs with parameters (\ref{para2}) and (\ref{para3}).
Since there are many constructions of the Latin square type SRGs, we only discuss the newness of the negative Latin
square type SRGs with the above parameters, namely, $\epsilon=-1$ in parameters (\ref{para2}) and (\ref{para3}).

One may check the known constructions of the SRGs with the
parameters (\ref{para2}), (\ref{para3}) via the online database
\cite{brouwer}. It is well known that projective two-weight codes
can be used to construct SRGs (\cite{calderbank-kantor}), one can
check the known constructions of two-weight codes via the online
database \cite{ericdatabase}. Next we give the following
constructions of SRGs with parameters (\ref{para2}) and
(\ref{para3}).

\begin{result}
    \cite{calderbank-kantor}
    \label{FE1}
    Let $k = 2m$ and $\mathcal{Q}$ be a non-degnerate quadratic form on $\ef_q$ with $q$ odd. Let
    $$M =\{v\in \ef_q^k\setminus\{0\} | \tr(\mathcal{Q}(v))\ \mbox{are non-zero squares}\},$$
    and $$M^\prime =\{v\in \ef_q^k\setminus\{0\} | \tr(\mathcal{Q}(v))\ \mbox{are squares}\}.$$ Then:

    $(1)$ the Cayley graph generated by $M$ in  $(\ef_q^k,+)$ is a
    $(p^{2k},\frac{1}2(p^k-p^{k-1})(p^k-\epsilon),\frac{1}4(p^k-p^{k-1})^2-
    \frac{3\epsilon}2(p^k-p^{k-1})+p^k\epsilon,\frac{1}2(p^k-p^{k-1})(\frac{1}2(p^k-p^{k-1})-\epsilon))$
    strongly regular graph, where $\varepsilon=\pm 1$ and depends on $\mathcal{Q}$.

    $(2)$ the Cayley graph generated by $M^\prime$ in  $(\ef_q^k,+)$ is a
    $(p^{2k},\frac{1}2(p^k+p^{k-1}+2\epsilon)(p^k-\epsilon),\frac{1}4(p^k+p^{k-1}+2\epsilon)^2-
    \frac{3\epsilon}2(p^k+p^{k-1}+2\epsilon)+p^k\epsilon,\frac{1}4(p^k+p^{k-1})(p^k+p^{k-1}+2\epsilon))$
    strongly regular graph, where $\varepsilon=\pm 1$ and depends on $\mathcal{Q}$.
\end{result}

The SRGs constructed by Result \ref{FE1} (1) are called
\textit{FE1}, and the SRGs from Result \ref{FE1} (2) are called
\textit{RT2} in \cite{calderbank-kantor}. They both are also
called \textit{affine polar graphs} in \cite[P.
852]{handbookcomb}. To the best of our knowledge, affine polar
graphs are the only known infinitive construction of the SRGs with
parameters (\ref{para2}) and (\ref{para3}) when $p\ge 5$.

For small parameters, now we discuss the known constructions of
the SRGs with parameters (\ref{para2}) and (\ref{para3}).

When $p=3$, by Theorem \ref{thm2} and the bent function in Result
\ref{binomialbent}, in field $\ef_{3^8}$,  we get an SRG with
parameter $(6561,2214,729,756)$. Computed by MAGMA, the order of
its automorphism group and the $3$-rank of the adjacent matrix of
the SRG are $(2^4\cdot3^8, 566)$. By comparing to \cite[Table
4]{tan-pott-srgbent}, we know that this graph is new.

When $p=5$, in the field $\ef_{5^4}$, the known constructions of
the SRGs with parameters $(625,260,105,110),(625,364,213,210)$ are
affine polar graphs, or from Theorems \ref{thm2},\ref{thm3}, or
from the projective two weight codes in Chen (\cite{ericreport}).
It is verified that Chen's SRGs are isomorphic to affine polar
graphs and the SRGs from the bent function in Result
\ref{binomialbent} are new.

When $p=7$, in the field $\ef_{7^4}$, the known constructions of
the SRGs with parameters $(2401,1050,455,462),(2401,1350,761,756)$
are the same as the case $p=5$. Using MAGMA, it is verified that
Chen's SRGs (\cite{ericreport}) are also isomorphic to affine
polar graphs and the SRGs from the bent function in Result
\ref{binomialbent} are new.

In the following two tables, we give some computational results of
the SRGs from different constructions. In the first column, we
list the parameters of the SRGs. The group $Aut(\mathcal{G})$ is
the full automorphism group of the SRG $\mathcal{G}$, and $M$
denotes an adjacency matrix of $\mathcal{G}$, to be considered in
$\ef_p$. The abbreviation $n.L.$ means the SRG is of negative
Latin square type. The symbol $\heartsuit$ means the SRG is
constructed by the bent function in Result \ref{binomialbent}.


\begin{center}
\tabcaption{SRGs with parameters (\ref{para2})}
\label{table1}
\begin{tabular}{c|c|c|c|c} \hline
 $(v,k,\lambda,\mu)$                 &Type     & Rank of $M$   &$|Aut(\mathcal{G})|$                 &note\\ \hline   \hline
 $(625,260,105,110)$                 &n.L.     &86             &$2^6\cdot3\cdot5^6\cdot13$           &affine polar        \\ \hline
 $(625,260,105,110)\heartsuit$       &n.L.     &104            &$2^4\cdot5^4$                        &new   \\ \hline
 $(2401,1050,455,462)$               &n.L.     &237            &$2^6\cdot3^2\cdot5^2\cdot7^6$        &affine polar                \\ \hline
 $(2401,1050,455,462)\heartsuit$     &n.L.     &335            &$2^3\cdot3\cdot7^4$                  &new \\ \hline
 \end{tabular}
\end{center}

\begin{center}
\tabcaption{SRGs with parameters (\ref{para3})}
\label{table2}
\begin{tabular}{c|c|c|c|c} \hline
 $(v,k,\lambda,\mu)$                 &Type     & Rank of $M$   &$|Aut(\mathcal{G})|$                 &note\\ \hline   \hline
 $(625,364,213,210)$                 &n.L.     &625            &$2^6\cdot3\cdot5^6\cdot13$           &affine polar        \\ \hline
 $(625,364,213,210)\heartsuit$       &n.L.     &625            &$2^4\cdot5^4$                        &new   \\ \hline
 $(2401,1350,761,756)$               &n.L.     &2401           &$2^6\cdot3^2\cdot5^2\cdot7^6$        &affine polar                \\ \hline
 $(2401,1350,761,756)\heartsuit$     &n.L.     &2401           &$2^3\cdot3\cdot7^4$                  &new  \\ \hline
 \end{tabular}
\end{center}

\begin{remark}
    From Tables \ref{table1} and \ref{table2}, we conjecture that the bent functions in Result \ref{binomialbent} can give a family of
    SRGs of negative Latin square type which are not isomorphic to the affine polar graphs. It is difficult to prove it in general cases.

    We may see that the automorphism groups of the Cayley graphs generated by the PDSs in (\ref{graphs}) have the subgroup of
    order $p^{2k}$ coming from translations $\tau_c:x\mapsto x+c$ for any $c\in\ef_{p^{2k}}$ and the subgroup of order $2k$ coming from
    the Galois automorphism of $\ef_{p^{2k}}$. This means that $2kp^{2k}$ divides the order of the automorphism groups of the Cayley graphs
    of (\ref{graphs}). For the bent functions given by Result (\ref{binomialbent}), we conjecture that $|Aut(\mathcal{G}(X))|$ is not divisible by $p^{2k+1}$,
    where $X=D_\mathcal{S}\ \mbox{or}\ D_\mathcal{S}^\prime$. This is a possible method to prove that they are new.
\end{remark}

\begin{acknowledgements}
Research supported by the National Research Foundation of Singapore under Research
Grant NRF-CRP2-2007-03 and by the Nanyang Technological University under Research Grant
M58110040. The authors would like to thank Professor Alexander Pott for helpful discussions.
They are also indebted to one of the anonymous referees to point out the result in Proposition 1 and the
two subgroups of the automorphism groups of the SRGs in this paper.
\end{acknowledgements}


\end{document}